\documentclass[11pt,reqno]{amsart}

\usepackage[T1]{fontenc}
\usepackage{lmodern}
\usepackage{microtype}
\usepackage{mathtools,amssymb}
\usepackage{enumitem}
\usepackage{aliascnt}
\usepackage[hidelinks]{hyperref}
\hypersetup{
  pdftitle={Ideals, Well-orderable Families, and Baire category in Truss's Feferman-type model},
  pdfauthor={Jason Zesheng Chen},
  pdfsubject={Baire category and Truss model},
  pdfkeywords={axiom of choice, dependent choice, perfect set property, Cohen forcing, Baire category, Rothberger property, strong measure zero, Marczewski ideal, Sacks forcing}
}
\usepackage[nameinlink,noabbrev]{cleveref}

\allowdisplaybreaks
\setlist[enumerate,1]{label=(\arabic*),leftmargin=2.4em,itemsep=.25em,topsep=.4em}
\setlist[itemize]{leftmargin=1.8em,itemsep=.2em,topsep=.4em}

\newcommand{\R}{\mathbb R}
\newcommand{\WO}{\mathsf{WO}}
\newcommand{\SN}{\mathcal{SN}}
\newcommand{\UN}{\mathcal{UN}}
\newcommand{\Mar}{s^0}
\newcommand{\cM}{\mathcal M}
\newcommand{\cN}{\mathcal N}
\newcommand{\Add}{\operatorname{Add}}
\newcommand{\Fn}{\operatorname{Fn}}
\newcommand{\RO}{\operatorname{RO}}
\newcommand{\Ult}{\operatorname{Ult}}
\newcommand{\diam}{\operatorname{diam}}
\newcommand{\restr}{\mathbin{\upharpoonright}}
\newcommand{\BCTWO}{\mathsf{BCT}_{\WO}}
\newcommand{\RSWO}{\mathsf{RS}_{\WO}}
\newcommand{\TP}{\mathsf{TP}_{\WO}}
\newcommand{\meq}{=^{\infty}}

\theoremstyle{plain}
\newtheorem{theorem}{Theorem}[section]
\newaliascnt{proposition}{theorem}
\newtheorem{proposition}[proposition]{Proposition}
\aliascntresetthe{proposition}
\newaliascnt{lemma}{theorem}
\newtheorem{lemma}[lemma]{Lemma}
\aliascntresetthe{lemma}
\newaliascnt{corollary}{theorem}
\newtheorem{corollary}[corollary]{Corollary}
\aliascntresetthe{corollary}

\theoremstyle{definition}
\newaliascnt{definition}{theorem}
\newtheorem{definition}[definition]{Definition}
\aliascntresetthe{definition}

\theoremstyle{remark}
\newaliascnt{remark}{theorem}
\newtheorem{remark}[remark]{Remark}
\aliascntresetthe{remark}
\newaliascnt{question}{theorem}
\newtheorem{question}[question]{Question}
\aliascntresetthe{question}

\crefname{theorem}{theorem}{theorems}
\Crefname{theorem}{Theorem}{Theorems}
\crefname{proposition}{proposition}{propositions}
\Crefname{proposition}{Proposition}{Propositions}
\crefname{lemma}{lemma}{lemmas}
\Crefname{lemma}{Lemma}{Lemmas}
\crefname{corollary}{corollary}{corollaries}
\Crefname{corollary}{Corollary}{Corollaries}
\crefname{definition}{definition}{definitions}
\Crefname{definition}{Definition}{Definitions}
\crefname{remark}{remark}{remarks}
\Crefname{remark}{Remark}{Remarks}
\crefname{question}{question}{questions}
\Crefname{question}{Question}{Questions}

\title[Baire Category in Truss Model]
{Ideals, well-orderable families, and Baire category in Truss's Feferman-type model}

\author{Jason Zesheng Chen}

\begin{document}

\begin{abstract}
Let $\BCTWO$ assert that every well-orderable family of dense open subsets
of a perfect Polish space has dense intersection, and let $\RSWO$ assert that for every nonempty countable partial order,
there is a filter meeting every member of any given well-orderable family of dense subsets.  Over $\mathsf{ZF}$ they are shown to be equivalent. We show that in Truss's model, they characterize those $A\subseteq2^\omega$ for
which every $A$-indexed family of dense open sets has dense intersection.
Assuming choice for well-orderable families of nonempty sets, we obtain an analogous characterization for total relations $R\subseteq X\times Y$ with meager vertical sections, whenever $Y$ is a surjective image of $2^\omega$.

The Feferman-type model $\mathfrak N_{\aleph_1}$ studied by Truss \cite{Truss1974} satisfies these hypotheses and $\mathsf{DC}$.
We study the ideal of well-orderable subsets of $2^\omega$ and its relations to other ideals in this model. We show, among other things: the well-orderable subsets of $2^\omega$ are exactly the sets which are
Rothberger in every finite power, have strong measure zero, are universally
null, are Marczewski null, or contain no perfect subset. The ideal of well-orderable subsets of $2^\omega$ is closed under
well-ordered unions and is incomparable with the meager ideal.  The
equivalence with the Rothberger property does not extend to the Hurewicz property.  Further more, in this model every set is Marczewski measurable, arbitrary maps into separable metric spaces have continuous perfect restrictions, and every set of positive outer measure contains a perfect subset.
\end{abstract}

\maketitle

\section{Introduction}\label{sec:intro}

A standard source of regularity for arbitrary sets of reals is the Solovay
model.  Starting from an inaccessible cardinal, one obtains a model of
$\mathsf{ZF}+\mathsf{DC}$ in which every set of reals is Lebesgue measurable
and has the perfect set property \cite{Solovay1970}.  A weaker question asks only whether every set of positive Lebesgue outer
measure must contain a perfect subset, and whether an inaccessible cardinal
is still needed.  An equivalent question was asked, for instance, by Liang Yu on MathOverflow\footnote{Specifically, Yu asked whether the consistency of $\mathsf{ZF}+\mathsf{DC}$ plus ``every non-null set has a perfect subset'' implies the consistency of an inaccessible cardinal. The author, then being a fledgling grad student, misunderstood the meaning of \textit{non-null} and replied with a laughably not-even-wrong answer, a situation that this article seeks to rectify.}
\cite{YuMO2021}.

It turns out that an inaccessible cardinal is not needed. The relevant model is the
Feferman-type model $\mathfrak N_{\aleph_1}$ studied by Truss
\cite{Truss1974}.  Recall that this model is obtained from an $\Add(\omega,\omega_1)$-generic sequence
$\langle c_\xi:\xi<\omega_1\rangle$ of Cohen reals over $L$ by taking
the least model of $\mathsf{ZF}$ containing $L$ and every proper initial
segment $\langle c_\xi:\xi<\alpha\rangle$, $\alpha<\omega_1$.  Truss characterized the well-orderable sets of reals in that model: a well-orderable set of reals is contained
in the reals of some bounded intermediate extension, while a set which is not well-orderable contains a perfect subset.

We observe that the theory in question holds in $\mathfrak N_{\aleph_1}$. To see this, note that a fresh Cohen real makes the
reals of the preceding intermediate extension null\footnote{see e.g., \cite{Cardona2020}. Alternatively: the real line is provably the union of a meager set $A$ and a null set $B$, and this partition is absolute across forcing extensions. If $c$ is Cohen, then for every ground model real $r$, we have $r\in c-B$. Since measure is preserved by translation, $c-B$ is null, so the ground model reals are a subset of a null set.}, so every well-orderable set in $M$ is null; Truss's dichotomy then gives a perfect subset of every non-null set.  The purpose of this paper is to place that observation in a broader Baire-category setting and to derive its consequences for standard smallness and selection properties.

We use the following abbreviations.  The statement $\TP$ is Truss's
dichotomy
\[
 A\subseteq2^\omega\quad\Longrightarrow\quad
 A\text{ is well-orderable, or }A\text{ contains a nonempty perfect set}.
\]
The statement $\BCTWO$ says that every well-orderable family of dense open subsets of a perfect Polish space has dense intersection.  We write $\mathsf{AC}_{\WO}$ for choice for well-orderable families of nonempty sets. 

The first part of the paper is independent of the particular symmetric model.  Over $\mathsf{ZF}$, $\BCTWO$ is equivalent to the statement $\RSWO$ that for every nonempty countable partial order,
there is a filter meeting every member of any given well-orderable family of dense subsets.  If $\TP$ is added, then for every
$A\subseteq2^\omega$ and every perfect Polish space $X$,
\begin{equation}\label{eq:intro-baire}
 A\text{ is well-orderable}
 \quad\Longleftrightarrow\quad
 \begin{gathered}
  \text{every $A$-indexed family of dense open subsets}\\
  \text{of $X$ has dense intersection.}
 \end{gathered}
\end{equation}
The converse is witnessed by a perfect subset of $A$ indexing complements of singletons.  Under $\mathsf{AC}_{\WO}$ the same argument applies to total relations whose vertical sections are meager, provided the set indexing the vertical sections is a surjective image of $2^\omega$.

We will then show that Truss's model satisfies $\BCTWO$.  The previous results then give well-orderability characterizations involving eventual domination, infinite agreement, simultaneous splitting, and arbitrary countable covers.  

The covering consequences identify several familiar smallness properties.
For every $A\subseteq2^\omega$ in $M$,
\begin{equation}\label{eq:intro-smallness}
 \begin{gathered}
  A\text{ is well-orderable}
  \quad\Longleftrightarrow\quad A\text{ is Rothberger}
  \quad\Longleftrightarrow\quad A\in\SN\\
  \Longleftrightarrow\quad A\in\UN
  \quad\Longleftrightarrow\quad A\in\Mar
  \quad\Longleftrightarrow\quad A\text{ contains no perfect subset}.
 \end{gathered}
\end{equation}
The equivalent conditions also include that every finite power of $A$ is Rothberger.  The resulting ideal is closed under ordinal-indexed unions, but it is not the meager ideal: $2^\omega\cap L$ is a well-orderable nonmeager set without the Baire property.  Nor does the equivalence extend to the
Hurewicz property.

The same hypotheses imply that every subset of $2^\omega$ is Marczewski measurable. A countable fusion argument gives continuous restrictions of arbitrary maps into separable metric spaces; on a further perfect subset, such a restriction can be made either constant or a topological embedding. Finally, it's not hard to see that the forcing of non-well-orderable subsets of $2^\omega$, ordered by inclusion modulo the ideal of well-orderable sets, is forcing-equivalent to Sacks forcing.

The paper separates the classical construction from the consequences proved from it.  The fixed-cardinal versions of several category and diagonalization arguments are standard, and continuous restriction theorems for Marczewski-measurable functions are classical.  The contribution here is the well-orderable-indexed formulation, its verification in Truss's model, and the resulting collection of equivalences and separations.

\section{ZF-provable results}\label{sec:abstract}

We work in $\mathsf{ZF}$ unless stronger assumptions are stated.  In what follows, a set is \emph{countable} if it injects into $\omega$; for a nonempty set this is equivalent in $\mathsf{ZF}$ to admitting an enumeration by $\omega$.  By a Polish space we mean a separable completely metrizable space.  When a concrete presentation is needed, we fix a compatible complete metric and a countable dense sequence, and hence a regular countable base.  Such a presentation may be coded by a real. Perfect sets and perfect Polish spaces are assumed nonempty. 

We record two standard Polish-space facts whose usual constructions require no choice.

\begin{lemma}\label{lem:polish-coding}
Every nonempty Polish space is the range of a Borel map from $2^\omega$. And every nonempty open subset of a perfect Polish space contains a copy of $2^\omega$.
\end{lemma}

The usual proofs are choice-free once a complete metric and a countable dense set are fixed; see \cite{Kechris1995}.

\begin{definition}\label{def:principles}
The statement $\TP$ says that every subset of $2^\omega$ is either
well-orderable or contains a perfect subset.  The statement $\BCTWO$ says
that, for every perfect Polish space $X$, the intersection of every
well-orderable family of dense open subsets of $X$ is dense.  Finally,
$\RSWO$ says that every well-orderable family of dense subsets of a nonempty
countable partial order has a common filter.  Filters are upward closed toward
weaker conditions and downward directed.
\end{definition}

\begin{remark}\label{rem:choice-catalogue}
Despite resemblance to the Baire category theorem and Rasiowa-Sikorski lemma, neither $\BCTWO$ nor $\RSWO$ is a choice principle. In fact, $\BCTWO$ implies that $2^\omega$ is not well-orderable and is therefore incompatible with the axiom of choice.

Similarly indexed dense-set formulations have been studied in choiceless set theory. For example, Tachtsis considers $\mathsf{MA}(\kappa)$ in $\mathsf{ZF}$ for well-ordered cardinals $\kappa$ \cite{Tachtsis2016}.  Karagila and Schilhan define distributivity directly for families of dense open sets indexed by an arbitrary set $X$ \cite[Definition 3.1]{KaragilaSchilhan2023}.  The statements
in \Cref{def:principles} restrict the underlying spaces or partial orders and quantify over all well-orderable index sets, which is the form used below.
\end{remark}

\begin{lemma}\label{lem:continuum-not-wo}
The statement $\BCTWO$ implies that $2^\omega$ is not well-orderable.
Consequently, no nonempty perfect subset of $2^\omega$ is well-orderable.
\end{lemma}

\begin{proof}
If $2^\omega$ were well-orderable, then
\[
 \{2^\omega\smallsetminus\{x\}:x\in2^\omega\}
\]
would be a well-orderable family of dense open sets with empty intersection.
Every nonempty perfect subset of $2^\omega$ is homeomorphic to $2^\omega$.
\end{proof}

\begin{theorem}\label{thm:indexed-baire}
Assume $\TP+\BCTWO$.  Let $X$ be a perfect Polish space.  For every $A\subseteq2^\omega$, the following are equivalent:
\begin{enumerate}
\item $A$ is well-orderable;
\item for every assignment $a\mapsto D_a$ of a dense open subset of $X$ to
each $a\in A$, the intersection $\bigcap_{a\in A}D_a$ is dense in $X$;
\item every such intersection is nonempty.
\end{enumerate}
\end{theorem}

\begin{proof}
If $A$ is well-orderable, then the range $\{D_a:a\in A\}$ is a
well-orderable family, so (2) follows from $\BCTWO$.  The implication
(2)$\Rightarrow$(3) is immediate.

Suppose that $A$ is not well-orderable.  By $\TP$, choose a perfect
$P\subseteq A$.  Fix a homeomorphism $h:P\to2^\omega$ and, by
\Cref{lem:polish-coding}, a surjection $\pi:2^\omega\to X$.  Define
\[
 D_a=
 \begin{cases}
 X\smallsetminus\{\pi(h(a))\},&a\in P,\\
 X,&a\notin P.
 \end{cases}
\]
Every $D_a$ is dense open because $X$ has no isolated points, while the
subfamily indexed by $P$ omits every point of $X$.  Hence the full
intersection is empty.
\end{proof}

\begin{theorem}\label{thm:bct-rs}
Over $\mathsf{ZF}$, the statements $\BCTWO$ and $\RSWO$ are equivalent.
\end{theorem}

\begin{proof}
Assume first $\BCTWO$, and let $Q$ be a nonempty countable partial order.
Adjoin a largest condition if necessary.  Let $\bar Q$ be the separative
quotient of $Q$, and let
\[
 e:Q\longrightarrow\RO(\bar Q)^+
\]
be the canonical dense map.  Let $B$ be the countable Boolean subalgebra of
$\RO(\bar Q)$ generated by $e[Q]$.  These constructions can be made
recursively from an enumeration of $Q$; see \cite[Chapter 14]{Jech2003}.

Every countable Boolean algebra has an ultrafilter in $\mathsf{ZF}$.  Indeed,
enumerate the algebra and decide successively whether to include each element
or its complement while keeping the finite meet nonzero.  Thus the Stone
space $X=\Ult(B)$ is nonempty.  For $b\in B$, write
$[b]=\{u\in X:b\in u\}$.

Let $D\subseteq Q$ be dense, and let $p,q\in Q$.  Put
\[
 O_D=\bigcup_{d\in D}[e(d)]
\]
and
\[
 H_{p,q}=X\smallsetminus[e(p)\wedge e(q)]
 \ \cup\!\bigcup_{r\leq p,q}[e(r)].
\]
Both sets are dense open.  For $O_D$, given nonzero $b\in B$, choose
$p_0\in Q$ with $0<e(p_0)\leq b$ and then $d\in D$ with $d\leq p_0$.
For $H_{p,q}$, if
$b\wedge\neg(e(p)\wedge e(q))$ is nonzero, the corresponding clopen set
works.  Otherwise choose $s\in Q$ with $0<e(s)\leq b$.  The inequalities
$e(s)\leq e(p),e(q)$ say that every extension of $s$ is compatible with
$p$ and with $q$.  Choose $t\leq s,p$ and then $r\leq t,q$.  The nonempty
clopen set $[e(r)]$ is contained in $[b]\cap H_{p,q}$.

Let $\mathcal D$ be a well-orderable family of dense subsets of $Q$.  If $B$
has an atom, its principal ultrafilter belongs to every dense open subset of
$X$.  If $B$ is atomless, then $X$ is a perfect Polish space,
and $\BCTWO$ gives an ultrafilter belonging to every member of
\[
 \{O_D:D\in\mathcal D\}\cup\{H_{p,q}:p,q\in Q\}.
\]
In either case, fix such an ultrafilter $u$ and let
\[
 G=\{p\in Q:e(p)\in u\}.
\]
The sets $O_D$ ensure that $G$ meets every $D\in\mathcal D$.  If
$p,q\in G$, then $u\in[e(p)\wedge e(q)]\cap H_{p,q}$, so some
$r\leq p,q$ belongs to $G$.  Upward closure toward weaker conditions is
immediate.  Thus $G$ is the required filter.

Conversely, assume $\RSWO$.  Let $X$ be a perfect Polish
space, let $\mathcal D$ be a well-orderable family of dense open subsets of
$X$, and fix a nonempty open set $U\subseteq X$.  Consider the countable
partial order of finite sequences
\[
 \langle B_0,\dots,B_n\rangle
\]
of nonempty basic open sets satisfying
\[
 \overline{B_0}\subseteq U,\qquad
 \overline{B_{i+1}}\subseteq B_i,\qquad
 \diam(B_i)<2^{-i}
\]
for all applicable $i$, ordered by end extension.  For each
$D\in\mathcal D$, the set of conditions whose last basic open set has closure
contained in $D$ is dense.  The sets requiring length at least $n$, for
$n<\omega$, are dense as well.  These requirements form a well-orderable
family, so $\RSWO$ supplies a common filter.

Since compatible conditions in this tree are comparable, the filter determines
a branch $\langle B_n:n<\omega\rangle$.  The distinguished centers of the
$B_n$ form a Cauchy sequence.  Completeness gives a limit $x$, and the closure
nesting gives $x\in U\cap B_n$ for every $n$.  Whenever the requirement for
$D$ is met, some $\overline{B_n}$ is contained in $D$, and therefore
$x\in D$.  Thus $U\cap\bigcap\mathcal D$ is nonempty.  Since $U$ was
arbitrary, the intersection is dense.
\end{proof}

\begin{corollary}\label{cor:indexed-rs}
Assume $\TP+\BCTWO$.  For every $A\subseteq2^\omega$, the following are equivalent:
\begin{enumerate}
\item $A$ is well-orderable;
\item whenever $Q$ is a nonempty countable partial order and
$\langle D_a:a\in A\rangle$ is an $A$-indexed family of dense subsets of
$Q$, there is a filter meeting every $D_a$.
\end{enumerate}
\end{corollary}

\begin{proof}
The forward implication follows from \Cref{thm:bct-rs}.  Conversely, suppose
that $A$ is not well-orderable.  Choose a perfect $P\subseteq A$ and a
homeomorphism $h:P\to2^\omega$.  In $Q=2^{<\omega}$, ordered by extension,
put
\[
 D_a=\{s\in2^{<\omega}:s\not\subseteq h(a)\}\quad(a\in P),
 \qquad D_a=Q\quad(a\notin P).
\]
Every $D_a$ is dense.  The conditions in a filter on $2^{<\omega}$ are
pairwise compatible, so their union is a finite string or a real.  Extend it,
if necessary, to $y\in2^\omega$ and let $a=h^{-1}(y)$.  Every condition in
the filter is an initial segment of $y$, so the filter misses $D_a$.
\end{proof}

The passage from dense open sets to arbitrary meager sets uses choice only to
select decompositions for a well-orderable family.

\begin{corollary}\label{cor:indexed-meager}
Assume $\mathsf{AC}_{\WO}+\TP+\BCTWO$.  Let $A\subseteq2^\omega$ and let
$X$ be a perfect Polish space.  The following are equivalent:
\begin{enumerate}
\item $A$ is well-orderable;
\item for every $A$-indexed family $\langle M_a:a\in A\rangle$ of meager
subsets of $X$, the complement of $\bigcup_{a\in A}M_a$ is dense;
\item no $A$-indexed family of meager subsets of $X$ covers $X$.
\end{enumerate}
\end{corollary}

\begin{proof}
Suppose that $A$ is well-orderable.  By $\mathsf{AC}_{\WO}$, choose for each
$a\in A$ a sequence $\langle F_{a,n}:n<\omega\rangle$ of closed nowhere
dense sets covering $M_a$.  The set $A\times\omega$ is well-orderable, so
$\BCTWO$ applied to the dense open sets $X\smallsetminus F_{a,n}$ gives (2).
The implication (2)$\Rightarrow$(3) is immediate.

If $A$ is not well-orderable, use a perfect $P\subseteq A$ and a surjection
$P\to X$ as in \Cref{thm:indexed-baire}.  The corresponding singletons,
together with empty sets off $P$, form an $A$-indexed meager cover of $X$.
\end{proof}

The conclusion in (2) does not assert that the union is meager.  This
distinction will be witnessed in Truss's model by the constructible reals.

\begin{theorem}[Meager-section theorem]\label{thm:meager-section}
Assume $\mathsf{AC}_{\WO}+\TP+\BCTWO$.  Let $X$ be a nonempty perfect
Polish space, let $Y$ be a nonempty set equipped with a surjection
$\pi:2^\omega\to Y$, and let $R\subseteq X\times Y$.  Suppose that every
vertical section
\[
 R^y=\{x\in X:xRy\}
\]
is meager and that $R$ is total:
\[
 \forall x\in X\;\exists y\in Y\quad xRy.
\]
Then for every $A\subseteq2^\omega$, the following are equivalent:
\begin{enumerate}
\item $A$ is well-orderable;
\item for every $f:A\to Y$ there is $x\in X$ such that
$\neg(xRf(a))$ for every $a\in A$.
\end{enumerate}
Moreover, if $A$ is well-orderable, then for every $f:A\to Y$ such points are dense in $X$.
\end{theorem}

\begin{proof}
If $A$ is well-orderable and $f:A\to Y$, apply
\Cref{cor:indexed-meager} to the family $\langle R^{f(a)}:a\in A\rangle$. Its union has dense complement, and every point in that complement avoids the entire family.

Conversely, suppose that $A$ is not well-orderable.  Choose a perfect $P\subseteq A$ and a homeomorphism $h:P\to2^\omega$.  Fix $y_0\in Y$ and put
\[
 f(a)=
 \begin{cases}
 \pi(h(a)),&a\in P,\\
 y_0,&a\notin P.
 \end{cases}
\]
Then $f[A]=Y$.  Totality says that every $x\in X$ is $R$-related to some member of $f[A]$, so no common avoider exists.
\end{proof}

\begin{proposition}\label{prop:uniform-sections}
Assume $\TP+\BCTWO$.  In \Cref{thm:meager-section}, replace the hypothesis
that every $R^y$ is meager by a fixed representation
\[
 R=\bigcup_{n<\omega}R_n
\]
such that every $(R_n)^y$ is closed nowhere dense.  Then the same equivalence
holds without $\mathsf{AC}_{\WO}$.
\end{proposition}

\begin{proof}
For well-orderable $A$ and $f:A\to Y$, apply $\BCTWO$ to the dense open
sets $X\smallsetminus(R_n)^{f(a)}$, indexed by $A\times\omega$.  The converse is
unchanged.
\end{proof}

\begin{remark}\label{rem:relations-literature}
Under the axiom of choice, relations of this kind give familiar cardinal characteristics by considering the least size of a set $B\subseteq Y$ such that every $x\in X$ is $R$-related to some member of $B$; see Blass \cite{Blass1996}.
\end{remark}

\section{Truss's model}\label{sec:model}

 Throughout this section, $\omega_1=\omega_1^L$.  Over $L$, let
\[
 \mathbb P=\Fn(\omega_1\times\omega,\omega,{<}\omega),
\]
let $G\subseteq\mathbb P$ be generic, and let $c_\xi\in\omega^\omega$ be the
$\xi$th Cohen real.  For $\alpha<\omega_1$, put
\[
 G_{<\alpha}=G\restr(\alpha\times\omega).
\]
We write the corresponding intermediate extensions as $L[G_{<\alpha}]$.

For $\alpha<\omega_1$, let $c_{<\alpha}=\langle c_\xi:\xi<\alpha\rangle$, so that $L[G_{<\alpha}]=L[c_{<\alpha}]$.
Following Truss \cite{Truss1974}, let $M=\mathfrak N_{\aleph_1}$ be the Feferman-type inner model of $L[G]$ obtained by adjoining every proper initial segment $c_{<\alpha}$, $\alpha<\omega_1$ but not the full sequence.

\begin{theorem}\label{thm:truss}
The model $M$ has the following properties.
\begin{enumerate}
\item $M\models\mathsf{ZF}+\mathsf{AC}_{\WO}$.
\item If $A\subseteq\R$ belongs to $M$, then
\[
A\text{ is well-orderable in }M
\quad\Longleftrightarrow\quad
\exists\alpha<\omega_1\quad A\subseteq\R^{L[G_{<\alpha}]}.
\]
\item Every non-well-orderable set of reals contains a perfect subset.
\item $M\models V=L(\R)$, and hence $M\models\mathsf{SVC}(\R)$.
\end{enumerate}
\end{theorem}

The facts collected in \Cref{thm:truss} are due to Truss and, for (4), Ryan-Smith.
The bounding used in (2) is the L\'evy-style argument of Truss \cite[Lemma 2.4, p.~202]{Truss1974}; the converse direction follows from Truss's Theorem 3.1 and its use in the proof of Theorem 3.3, where Truss notes that a set contained in the relevant $\mathsf{ZFC}$ submodel can be well-ordered, while otherwise Theorem 3.1 applies \cite[Theorem 3.1 and the proof of Theorem 3.3, especially p.~208]{Truss1974}.
Statement (3) is Truss's Theorem 3.2 \cite[Theorem 3.2]{Truss1974}, and Truss also proves (1) for this model.
Ryan-Smith \cite[Proposition 3.19]{RyanSmith2025} shows that $M\models V=L(\R)$ and consequently $M\models\mathsf{SVC}(\R)$.

Here $\mathsf{SVC}(\R)$ means that every nonempty set is a surjective image
of $\R\times\eta$ for some ordinal $\eta$.

\begin{lemma}\label{lem:same-reals}
The symmetric model $M$ and the full forcing extension $L[G]$ have the same
reals.
\end{lemma}

\begin{proof}
Every $\mathbb P$-name for a real is equivalent to a nice name using
countably many coordinates.  In $L$, every countable subset of
$\omega_1^L$ is bounded.  Thus every real of $L[G]$ belongs to some
$L[G_{<\alpha}]$ and hence to $M$.  The reverse inclusion is immediate.
\end{proof}

For completeness, we record a modern verification of dependent choice.

\begin{proposition}\label{prop:dc}
The model $M$ satisfies $\mathsf{DC}$.
\end{proposition}

\begin{proof}
By \Cref{thm:truss}, $M\models\mathsf{SVC}(\R)$.  Ryan-Smith's local
characterization of dependent choice says that under $\mathsf{SVC}(\R)$,
$\mathsf{DC}$ is equivalent to the assertion that every subtree of
$\R^{<\omega}$ with no maximal node has a branch
\cite[Proposition 3.1]{RyanSmith2025}.  If such a tree $T$ belongs to $M$,
then the choice extension $L[G]$ has a branch through $T$.  A sequence of
reals is coded by one real, and by \Cref{lem:same-reals} that code
already belongs to $M$.
\end{proof}

Since countable ordinals are coded by reals,
\[
 \omega_1^M=\omega_1^{L[G]}=\omega_1^L.
\]

\begin{corollary}\label{cor:reals-not-wo}
The set $2^\omega$ is not well-orderable in $M$.
\end{corollary}

\begin{proof}
If all reals were well-orderable, \Cref{thm:truss} would put them in
$L[G_{<\alpha}]$ for some $\alpha<\omega_1$, contrary to the fresh
coordinate $c_\alpha$.
\end{proof}
When comparing a Polish space across intermediate extensions, we work with a fixed countable presentation and code open sets relative to the associated countable basis.
Note that the denseness of an open set and the nowhere-denseness of a closed set is absolute in the codes between the intermediate extensions considered here.

\begin{theorem}\label{thm:model-principles}
The model $M$ satisfies
\[
\mathsf{ZF}+\mathsf{DC}+\mathsf{AC}_{\WO}+\TP+\BCTWO.
\]
\end{theorem}

\begin{proof}
Only $\BCTWO$ remains to be proved.
Let $X$ be a perfect Polish space in $M$, and let $\mathcal D$ be a well-orderable family of dense open subsets of $X$.
Fix a complete compatible metric and a countable dense sequence for $X$, and let $\langle U_n:n<\omega\rangle$ be the resulting countable basis.
For $D\in\mathcal D$, code $D$ by
\[
d_D=\{n<\omega:U_n\subseteq D\}.
\]
The set $\{d_D:D\in\mathcal D\}$ is a well-orderable set of reals, so by \Cref{thm:truss} there is $\alpha<\omega_1$ such that every $d_D$ belongs to $L[G_{<\alpha}]$.
Increasing $\alpha$ if necessary, we may also assume that the chosen presentation of $X$ belongs to $L[G_{<\alpha}]$.

Let $U$ be a nonempty basic open subset of $X$.
Over $L[G_{<\alpha}]$, the forcing of nonempty basic open subsets of $U$ ordered by reverse inclusion has a countable atomless dense suborder, and hence is forcing-equivalent to Cohen forcing.
The fresh coordinate $c_\alpha$ therefore yields a point $x\in U$ which is Cohen-generic for the topology of $X$ over $L[G_{<\alpha}]$.
Since each $d_D$ belongs to $L[G_{<\alpha}]$ and codes a dense open set there, $x\in D$ for every $D\in\mathcal D$.
Thus
\[
U\cap\bigcap\mathcal D\neq\varnothing.
\]
Since $U$ was arbitrary, $\bigcap\mathcal D$ is dense in $X$.
\end{proof}

By \Cref{thm:indexed-baire,thm:model-principles}, in $M$ a set $A\subseteq2^\omega$ is well-orderable if and only if every $A$-indexed family of dense open subsets of any nonempty perfect Polish space has dense intersection.
By \Cref{cor:indexed-rs}, this is also equivalent to the following: for every nonempty countable partial order $Q$, every $A$-indexed family of dense subsets of $Q$ is met by a single filter.

\begin{remark}\label{rem:model-hartogs}
The Hartogs number of $\R$ in $M$ is $\omega_2^M$.  Indeed, every
well-orderable set of reals is contained in $\R^{L[G_{<\alpha}]}$ for some
$\alpha<\omega_1$.  The forcing which adds $G_{<\alpha}$ is countable in
$L$, so this intermediate extension satisfies the continuum hypothesis;
hence every such set injects into $\omega_1$.  Conversely, assigning to each
countable ordinal its least constructible real code gives an injection
$\omega_1\to\R\cap L$.  Thus, within $M$, $\BCTWO$ is the Baire-category
assertion for families of at most $\aleph_1$ dense open sets, and $\RSWO$ is
the countable-poset fragment of $\mathsf{MA}_{\aleph_1}$.  We retain the
well-orderable-family formulation because the abstract results do not assume
this model-specific cardinal calculation.
\end{remark}

We also need the following standard fact about Cohen reals.

\begin{lemma}\label{lem:no-dominating-real}
For every $\alpha<\omega_1$ and every
$g\in\omega^\omega\cap L[G_{<\alpha}]$, there is
$f\in\omega^\omega\cap L$ such that $f\not\leq^*g$.
\end{lemma}

\begin{proof}
The case $\alpha=0$ is immediate.  For $\alpha>0$, the forcing
$\Fn(\alpha\times\omega,\omega,{<}\omega)$ is countable in $L$.  Let
$\dot g$ be a name for a member of $\omega^\omega$ and let $p$ be a
condition.  Enumerate all pairs $(q,N)$ with $q\leq p$ and $N<\omega$ as
$\langle(q_i,N_i):i<\omega\rangle$.  Recursively choose distinct
$n_i\geq N_i$, an extension $r_i\leq q_i$, and $m_i<\omega$ such that
\[
 r_i\Vdash\dot g(n_i)=m_i.
\]
Define $f\in\omega^\omega\cap L$ by $f(n_i)=m_i+1$ and $f(n)=0$ elsewhere.
For every $q\leq p$ and $N<\omega$, some extension of $q$ forces
$f(n)>\dot g(n)$ at a coordinate $n\geq N$.  It follows that
\[
 p\Vdash\check f\not\leq^*\dot g.
\]
Applying this construction below a condition in the generic filter proves the
claim.
\end{proof}

\section{Domination, agreement, and splitting}\label{sec:diagonal}

We now apply \Cref{prop:uniform-sections} to several familiar relations.  The
results in this section, except for the boundedness theorem
\Cref{thm:boundedness-threshold}, hold in every model satisfying the stated
abstract hypotheses.

\subsection{Eventual domination}

For $g,f\in\omega^\omega$, write $g\leq^*f$ if
$g(n)\leq f(n)$ for all but finitely many $n$.  For $k<\omega$, let
\[
 R_k=\{(g,f):\forall n\geq k\;g(n)\leq f(n)\}.
\]
For fixed $f$, the section $(R_k)^f$ is closed nowhere dense.  Moreover,
$\leq^*=\bigcup_kR_k$ is total, since $g\leq^*g$.

\begin{corollary}\label{cor:nondomination}
Assume $\TP+\BCTWO$.  For every $A\subseteq2^\omega$, the following are
equivalent:
\begin{enumerate}
\item $A$ is well-orderable;
\item for every $F:A\to\omega^\omega$ there is $g\in\omega^\omega$ such
that $g\not\leq^*F(a)$ for every $a\in A$.
\end{enumerate}
Consequently, every dominating family in $\omega^\omega$ is
non-well-orderable.
\end{corollary}

\begin{proof}
Apply \Cref{prop:uniform-sections}.  For the final assertion, apply (2) to
the identity map on an alleged well-orderable dominating family.
\end{proof}

In Truss's model, universal boundedness has a different characterization.
Since $L\models\mathsf{CH}$, fix a sequence
$\langle d_\xi:\xi<\omega_1\rangle$ in $L$ which is dominating in
$\omega^\omega\cap L$.

\begin{theorem}\label{thm:boundedness-threshold}
In $M$, for every $A\subseteq2^\omega$, the following are equivalent:
\begin{enumerate}
\item $A$ is countable;
\item for every $F:A\to\omega^\omega$ there is $g\in\omega^\omega$ such
that $F(a)\leq^*g$ for every $a\in A$.
\end{enumerate}
\end{theorem}

\begin{proof}
The empty case is trivial.  If $A$ is nonempty and countable, enumerate it
with repetitions as $A=\{a_i:i<\omega\}$ and put
\[
 g(n)=\max\{F(a_i)(n):i\leq n\}.
\]
Then $g$ eventually dominates every $F(a_i)$.

Conversely, suppose first that $A$ is well-orderable but uncountable.  It has
a subset $\{a_\xi:\xi<\omega_1\}$.  Define
$F(a_\xi)=d_\xi$ and extend $F$ arbitrarily to the rest of $A$.  If some
$g\in M$ bounded the image, then $g$ would dominate every constructible
function, contrary to \Cref{lem:no-dominating-real}.

If $A$ is not well-orderable, choose a perfect $P\subseteq A$.  By
\Cref{lem:polish-coding}, there is a surjection $P\to\omega^\omega$.
Extend it to a function on $A$.  Its range is all of $\omega^\omega$, which
is not eventually bounded by any one function.
\end{proof}

Thus countability characterizes index sets whose every image is bounded,
whereas well-orderability characterizes index sets whose every image is
non-dominating.  In particular, an uncountable well-orderable family may be
unbounded, but it cannot be dominating.

\subsection{Infinite agreement and arbitrary covers}

Write $g\meq f$ if $g(n)=f(n)$ for infinitely many $n$.  Let
\[
 E_k=\{(g,f):\forall n\geq k\;g(n)\neq f(n)\}.
\]
Each section $(E_k)^f$ is closed nowhere dense, and the relation of being
eventually different, $\bigcup_kE_k$, is total: $g$ is eventually different
from $n\mapsto g(n)+1$.

\begin{corollary}[Infinite agreement]\label{cor:agreement}
Assume $\TP+\BCTWO$.  For every $A\subseteq2^\omega$, the following are
equivalent:
\begin{enumerate}
\item $A$ is well-orderable;
\item for every $F:A\to\omega^\omega$ there is $g\in\omega^\omega$ such
that $g\meq F(a)$ for every $a\in A$.
\end{enumerate}
\end{corollary}

\begin{proof}
Apply \Cref{prop:uniform-sections} to the eventually-different relation.
\end{proof}

Under the axiom of choice, Miller proved that the least size of a family for
which (2) fails is the least cardinal at which the Baire category theorem
fails \cite{Miller1982}.  We shall use the following elementary reformulation.

\begin{lemma}\label{lem:cover-selection}
For any set $A$, the following are equivalent in $\mathsf{ZF}$:
\begin{enumerate}
\item for every $F:A\to\omega^\omega$ there is $g\in\omega^\omega$ which
is infinitely equal to every member of $F[A]$;
\item whenever
\[
 A=\bigcup_{m<\omega}U_{n,m}\qquad(n<\omega)
\]
is a sequence of enumerated countable covers by arbitrary subsets of $A$,
there is $g\in\omega^\omega$ such that
\[
 \forall a\in A\;\exists^\infty n\quad a\in U_{n,g(n)}.
\]
\end{enumerate}
\end{lemma}

\begin{proof}
Given the covers, put
\[
 F(a)(n)=\min\{m:a\in U_{n,m}\}.
\]
A function infinitely equal to every $F(a)$ gives the required selections.
Conversely, given $F$, apply (2) to
$U_{n,m}=\{a\in A:F(a)(n)=m\}$.
\end{proof}

\begin{corollary}\label{cor:cover-selection}
Assume $\TP+\BCTWO$.  For $A\subseteq2^\omega$, $A$ is well-orderable if
and only if it satisfies the selection property in
\Cref{lem:cover-selection}(2).
\end{corollary}

\subsection{Simultaneous splitting}

For $x\subseteq\omega$ and $y\in[\omega]^\omega$, say that $x$ \emph{splits}
$y$ if both $x\cap y$ and $y\smallsetminus x$ are infinite.  Let $xRy$ mean that
$x$ does not split $y$.  Then
\[
 xRy\quad\Longleftrightarrow\quad
 y\subseteq^*x\ \lor\ y\subseteq^*(\omega\smallsetminus x).
\]
For $k<\omega$, each of the sets
\[
 \{x:\forall n\in y\smallsetminus k\;(n\in x)\},
 \qquad
 \{x:\forall n\in y\smallsetminus k\;(n\notin x)\}
\]
is closed nowhere dense in $2^\omega$, and their countable union is the
non-splitting section.  The relation is total: for every $x\subseteq\omega$,
either $x$ or its complement is infinite, and an infinite subset of that side
is not split by $x$.

\begin{corollary}[Simultaneous splitting]\label{cor:splitting}
Assume $\TP+\BCTWO$.  For every $A\subseteq2^\omega$, the following are
equivalent:
\begin{enumerate}
\item $A$ is well-orderable;
\item every function $F:A\to[\omega]^\omega$ has a common splitter.
\end{enumerate}
Consequently, every reaping family is non-well-orderable.
\end{corollary}

\begin{proof}
Apply \Cref{prop:uniform-sections} to the non-splitting relation.
\end{proof}

\begin{remark}\label{rem:perfect-witnesses}
Under $\TP$, the last assertions of \Cref{cor:nondomination,cor:splitting}
can be strengthened.  After the standard coding into $2^\omega$, every dominating or reaping family contains a Cantor set and is therefore in bijection with $2^\omega$.  This observation is confined to these real-coded examples; the abstract parameter set $Y$ in \Cref{thm:meager-section} carries no topology.
\end{remark}

\section{Covering and smallness properties}\label{sec:smallness}

A space $X$ is \emph{Rothberger} if, for every sequence
$\langle\mathcal U_n:n<\omega\rangle$ of open covers of $X$, one can choose
$U_n\in\mathcal U_n$ so that $X=\bigcup_nU_n$.

\begin{lemma}\label{lem:countable-subcover}
Assume $\mathsf{AC}_\omega$.  Every open cover of a subspace of a
second-countable space has a countable subcover.
\end{lemma}

\begin{proof}
Let $\mathcal B$ be a countable base and let $\mathcal U$ be an open cover of
$A$.  The family
\[
 \mathcal B'=\{B\in\mathcal B:\exists U\in\mathcal U\;(B\cap A\subseteq U)\}
\]
covers $A$.  For each $B\in\mathcal B'$, choose one corresponding $U_B$.
Only countable choice is used, and $\{U_B:B\in\mathcal B'\}$ covers $A$.
\end{proof}

\begin{proposition}\label{prop:rothberger}
Assume $\mathsf{DC}+\TP+\BCTWO$.  Every well-orderable subspace of
$2^\omega$ is Rothberger.  More strongly, from each sequence of open covers
one can select one member of each cover so that every point belongs to the
selected member for infinitely many indices.  Every finite power of a
well-orderable subspace is Rothberger.
\end{proposition}

\begin{proof}
The empty case is immediate, so assume that the space is nonempty.  Since
$\mathsf{DC}$ implies countable choice, use \Cref{lem:countable-subcover}
simultaneously for the given sequence of covers and enumerate the resulting
countable subcovers, with repetitions if necessary.  Apply
\Cref{cor:cover-selection}.  Finite powers of well-orderable sets are
well-orderable, and $(2^\omega)^n$ is homeomorphic to $2^\omega$.
\end{proof}

For $A\subseteq2^\omega$, say that $A$ has \emph{strong measure zero} if for
every $f\in\omega^\omega$ there are strings $s_n\in2^{f(n)}$ such that
$A\subseteq\bigcup_n[s_n]$.  It is \emph{universally null} if
$\mu^*(A)=0$ for every atomless Borel probability measure $\mu$ on
$2^\omega$.  Let $\SN$ and $\UN$ denote the corresponding ideals.  The
Marczewski ideal $\Mar$ consists of those sets $A$ such that every perfect
$P\subseteq2^\omega$ has a perfect subset disjoint from $A$.

\begin{lemma}\label{lem:roth-sn}
Every Rothberger subset of $2^\omega$ has strong measure zero.
\end{lemma}

\begin{proof}
Given $f\in\omega^\omega$, apply the Rothberger property to the finite open
covers $\{[s]:s\in2^{f(n)}\}$.
\end{proof}

\begin{lemma}\label{lem:sn-un}
In $\mathsf{ZF}+\mathsf{DC}$, every strong measure zero subset of $2^\omega$
is universally null.
\end{lemma}

\begin{proof}
Let $A$ have strong measure zero, let $\mu$ be atomless, and fix
$\varepsilon>0$.  Atomlessness and compactness give
\[
 \lim_{k\to\infty}\max\{\mu([s]):s\in2^k\}=0.
\]
Otherwise, for some $\delta>0$ the finitely branching tree of strings $s$
with $\mu([s])\geq\delta$ would have a branch $x$, and continuity from above
would give $\mu(\{x\})\geq\delta$.

Choose $k_n$ so that
\[
 \max\{\mu([s]):s\in2^{k_n}\}<\varepsilon2^{-n-1}.
\]
A strong measure zero cover $A\subseteq\bigcup_n[s_n]$, with
$s_n\in2^{k_n}$, then gives
\[
 \mu^*(A)\leq\sum_n\mu([s_n])<\varepsilon.
\]
\end{proof}

\begin{lemma}\label{lem:un-perfect}
A universally null set contains no perfect subset.
\end{lemma}

\begin{proof}
If $P\subseteq A$ is perfect, push the usual product measure on $2^\omega$
forward through a homeomorphism $2^\omega\to P$.  This gives an atomless
Borel probability measure on $2^\omega$ concentrated on $P$, for which
$\mu^*(A)=1$.
\end{proof}

\begin{theorem}[Equivalence of smallness properties]\label{thm:smallness}
Assume $\mathsf{ZF}+\mathsf{DC}+\TP+\BCTWO$.  For every
$A\subseteq2^\omega$, the following are equivalent:
\begin{enumerate}
\item $A$ is well-orderable;
\item $A$ is Rothberger;
\item $A^n$ is Rothberger for every $1\leq n<\omega$;
\item $A\in\SN$;
\item $A\in\UN$;
\item $A$ contains no perfect subset;
\item $A\in\Mar$.
\end{enumerate}
\end{theorem}

\begin{proof}
By \Cref{prop:rothberger}, (1) implies (2) and (3), and (3) implies (2).
By \Cref{lem:roth-sn,lem:sn-un,lem:un-perfect},
\[
 (2)\Longrightarrow(4)\Longrightarrow(5)\Longrightarrow(6).
\]
The statement $\TP$ gives (6)$\Rightarrow$(1).

If $A\in\Mar$, then $A$ contains no perfect set, so (7)$\Rightarrow$(6).
Conversely, suppose that $A$ is well-orderable and let $P$ be perfect.  Then
$A\cap P$ is well-orderable.  If $P\smallsetminus A$ were also well-orderable,
then $P$ would be well-orderable, contrary to \Cref{lem:continuum-not-wo}.
Thus $P\smallsetminus A$ is not well-orderable, and $\TP$ gives a perfect
$Q\subseteq P\smallsetminus A$.  Hence $A\in\Mar$.
\end{proof}

\begin{corollary}\label{cor:smallness-model}
In Truss's model, the well-orderable subsets of $2^\omega$ are exactly the
Rothberger sets, the sets all of whose finite powers are Rothberger, the
strong measure zero sets, the universally null sets, the Marczewski null
sets, and the sets containing no perfect subset.
\end{corollary}

For comparison, Scheepers and Tall proved that after adding uncountably many Cohen reals, every ground-model Lindel\"of space becomes Rothberger \cite{ScheepersTall2010}.
Here one obtains, inside $M$, the sharper characterization that a subset of $2^\omega$ is Rothberger exactly when it is well-orderable.
Cardona's result that one Cohen real makes the ground-model reals strong measure zero \cite{Cardona2020} likewise gives a direct forcing proof of the implication from well-orderability to strong measure zero.

\section{The ideal of well-orderable sets}\label{sec:ideal}

In this section we work in $M$ and let
\[
 \mathcal I=\{A\subseteq2^\omega:A\text{ is well-orderable}\}.
\]
By \Cref{cor:smallness-model}, $\mathcal I=\SN=\UN=\Mar$.

\begin{proposition}[Well-ordered unions]\label{prop:wellordered-unions}
If $\langle A_\xi:\xi<\lambda\rangle$ is an ordinal-indexed sequence of
members of $\mathcal I$, then
\[
 \bigcup_{\xi<\lambda}A_\xi\in\mathcal I.
\]
\end{proposition}

\begin{proof}
For each $\xi<\lambda$, the set of well-orders of $A_\xi$ is nonempty.  By
$\mathsf{AC}_{\WO}$, choose one for every $\xi$.  Well-order the union by
first taking the least index at which a point occurs and then using the
chosen well-order at that index.
\end{proof}

\begin{corollary}\label{cor:ideal-invariants}
There is no well-orderable family of members of $\mathcal I$ whose union is
$\mathcal I$-positive or covers $2^\omega$.  There is no well-orderable
$\mathcal I$-positive set and no well-orderable cofinal family in
$(\mathcal I,\subseteq)$.
\end{corollary}

\begin{proof}
A well-orderable family can be enumerated by an ordinal, so its union belongs
to $\mathcal I$ by \Cref{prop:wellordered-unions}.  A well-orderable positive
set is excluded by definition.  A cofinal family would cover every singleton
and hence all of $2^\omega$, contradicting the first assertion.
\end{proof}

Let $\cN$ and $\cM$ be the Lebesgue null and meager ideals, respectively.

\begin{proposition}\label{prop:ideal-comparisons}
In $M$:
\begin{enumerate}
\item $\mathcal I\subsetneq\cN$;
\item $\mathcal I$ and $\cM$ are incomparable.
\end{enumerate}
\end{proposition}

\begin{proof}
Since $\mathcal I=\UN$, every member of $\mathcal I$ is Lebesgue null.  The
perfect set
\[
 P_0=\{x\in2^\omega:\forall n\;x(2n)=0\}
\]
is null and nowhere dense, but does not belong to $\mathcal I$.  This proves
$\mathcal I\subsetneq\cN$ and $\cM\nsubseteq\mathcal I$.

For the other non-inclusion, let $C=2^\omega\cap L$.  The canonical well-order
of $L$ well-orders $C$, so $C\in\mathcal I$.  Suppose that $C$ were meager in
$M$.  A sequence of closed nowhere dense sets covering $C$ is coded by one
real, hence its code belongs to some $L[G_{<\alpha}]$.  Membership of a
constructible real in a coded closed set, and the assertion that a closed code
is nowhere dense, are absolute between this intermediate extension and $M$.
Thus $L[G_{<\alpha}]$ would see the constructible reals as meager.

If $\alpha=0$, this contradicts the Baire category theorem in $L$.  If
$\alpha>0$, then $\alpha$ is countable in $L$, so
$L[G_{<\alpha}]$ is forcing-equivalent over $L$ to an extension by one Cohen
real.  Kunen's theorem says that the ground-model reals remain nonmeager after
adding a Cohen real; see \cite[Theorem 1.4]{Cardona2020}.  This is again a
contradiction.  Hence $C\notin\cM$.
\end{proof}
\begin{corollary}\label{cor:constructible-nonbp}
The set $C=2^\omega\cap L$ is nonmeager in $M$ and does not have the Baire property.
\end{corollary}

\begin{proof}
Jech \cite[Lemma 26.50 and Example 26.52]{Jech2003} shows that after adding $\omega_1$ Cohen reals over $L$, the set of constructible reals is nonmeager and does not have the Baire property.
By \Cref{lem:same-reals}, $M$ and $L[G]$ have the same reals and the same coded open sets, so the same holds in $M$.
\end{proof}

Thus a well-orderable union of meager sets need not be meager: $C$ is the
well-orderable union of its singletons, although
\Cref{cor:indexed-meager} says that its complement is dense.

We next compare the Hurewicz property.  A space $X$ is \emph{Hurewicz} if,
for every sequence $\langle\mathcal U_n:n<\omega\rangle$ of open covers,
there are finite $\mathcal V_n\subseteq\mathcal U_n$ such that every point of
$X$ belongs to $\bigcup\mathcal V_n$ for all but finitely many $n$.

\begin{proposition}\label{prop:not-hurewicz}
There is a well-orderable set $A\subseteq2^\omega$ such that every finite power of $A$ is Rothberger, but $A$ is not Hurewicz.
\end{proposition}

\begin{proof}
Let $H=\omega^\omega\cap L$, and fix a homeomorphic embedding $e:\omega^\omega\to2^\omega$ coded in $L$.
Put $A=e[H]$.
The canonical well-order of $L$ well-orders $H$, and hence $A$, so \Cref{prop:rothberger} implies that every finite power of $A$ is Rothberger.

It remains to show that $H$ is not Hurewicz.
For $n,m<\omega$, let
\[
U_{n,m}=\{f\in H:f(n)\leq m\}.
\]
For each $n$, the increasing family $\{U_{n,m}:m<\omega\}$ is an open cover of $H$.
Suppose that finite subfamilies $\mathcal V_n\subseteq\{U_{n,m}:m<\omega\}$ witnessed the Hurewicz property, and define
\[
g(n)=\max\bigl(\{m:U_{n,m}\in\mathcal V_n\}\cup\{0\}\bigr).
\]
For every $f\in H$, eventual membership in $\bigcup\mathcal V_n$ implies $f\leq^*g$.
Thus $g\in M$ eventually dominates every member of $\omega^\omega\cap L$, contradicting \Cref{lem:no-dominating-real}.
Hence $H$ is not Hurewicz, and therefore neither is its homeomorphic copy $A$.
\end{proof}

\section{Marczewski measurability}
\label{sec:marczewski}

A set $A\subseteq2^\omega$ is \emph{Marczewski measurable}, or
$s$-measurable, if for every perfect $P\subseteq2^\omega$ there is a perfect
$Q\subseteq P$ such that either $Q\subseteq A$ or $Q\cap A=\varnothing$.

\begin{theorem}\label{thm:s-measurable}
Assume $\TP+\BCTWO$.  Every subset of $2^\omega$ is Marczewski measurable.
\end{theorem}

\begin{proof}
Let $A\subseteq2^\omega$ and let $P$ be perfect.  If $A\cap P$ is not
well-orderable, then $\TP$ gives a perfect subset of $A\cap P$.  Otherwise
$A\cap P$ is well-orderable.  The set $P\smallsetminus A$ cannot also be
well-orderable, by \Cref{lem:continuum-not-wo}; hence it contains a perfect
subset.
\end{proof}

The following fusion lemma gives a simultaneous form of Marczewski measurability for countably many sets.

\begin{lemma}\label{lem:simultaneous-fusion}
Assume $\mathsf{DC}$ and that every subset of $2^\omega$ is Marczewski
measurable.  If $\langle A_n:n<\omega\rangle$ is a sequence of subsets of
$2^\omega$ and $P$ is perfect, then there is a perfect $Q\subseteq P$ such
that $A_n\cap Q$ is clopen in $Q$ for every $n$.
\end{lemma}

\begin{proof}
Build by fusion perfect sets $K_s$, $s\in2^{<\omega}$, beginning with
$K_\varnothing=P$.  At stage $n$, for each $s\in2^n$, use Marczewski
measurability to choose a perfect $H_s\subseteq K_s$ which is either contained
in $A_n$ or disjoint from $A_n$.  Inside $H_s$, choose disjoint perfect sets
$K_{s^\frown0}$ and $K_{s^\frown1}$ lying in incompatible basic clopen sets,
with diameter below $2^{-n}$.  There are only finitely many choices at each
stage, and $\mathsf{DC}$ supplies the sequence of finite systems.

Put
\[
 Q=\bigcap_{n<\omega}\bigcup_{s\in2^n}K_s.
\]
The splitting and diameter requirements make $Q$ perfect.  For fixed $n$,
the pieces at the $(n+1)$th front are clopen in $Q$, and all descendants of
each piece make the same decision about $A_n$.  Hence $A_n\cap Q$ is a
finite union of clopen pieces.
\end{proof}

\begin{theorem}\label{thm:continuous-restriction}
Assume $\mathsf{DC}+\TP+\BCTWO$.  Let $Y$ be a separable metric space,
let $f:2^\omega\to Y$ be arbitrary, and let $P\subseteq2^\omega$ be perfect.
There is a perfect $Q\subseteq P$ such that $f\restr Q$ is continuous.
Countably many such functions admit a common perfect set of continuity inside
$P$.
\end{theorem}

\begin{proof}
Fix a countable base $\langle V_n:n<\omega\rangle$ for $Y$ and apply
\Cref{thm:s-measurable,lem:simultaneous-fusion} to the sets
$f^{-1}(V_n)$.  On the resulting perfect set $Q$, the preimage of every basic
open set is clopen, so $f\restr Q$ is continuous.  For countably many functions, use $\mathsf{DC}$ to choose countable bases for their codomains and apply the same fusion lemma to the countable collection of all corresponding
preimages.
\end{proof}

Continuous restrictions of Marczewski-measurable functions are classical; see Brown \cite{Brown1992} and the references therein.
The point here is that the argument above requires only $\mathsf{DC}$ together with Marczewski measurability of all subsets of $2^\omega$.

\begin{theorem}\label{thm:constant-embedding}
Assume $\mathsf{DC}+\TP+\BCTWO$.  Let $Y$ be a separable metric space,
let $f:2^\omega\to Y$ be arbitrary, and let $P$ be perfect.  There is a
perfect $Q\subseteq P$ such that $f\restr Q$ is either constant or a
topological embedding.
\end{theorem}

\begin{proof}
By \Cref{thm:continuous-restriction}, first pass to a perfect
$P_0\subseteq P$ on which $f$ is continuous.  The relation
\[
 E=\{(x,y)\in P_0^2:f(x)=f(y)\}
\]
is closed.  If $E$ has nonempty interior, choose nonempty relatively open
$U,V\subseteq P_0$ with $U\times V\subseteq E$.  Fix $v\in V$.  Then
$f(x)=f(v)$ for every $x\in U$, so $f$ is constant on a perfect subset of
$U$.

Otherwise $E$ is nowhere dense.  The usual Mycielski fusion construction
gives a perfect $Q\subseteq P_0$ whose distinct points are $E$-inequivalent
\cite{Mycielski1964}.  Thus $f\restr Q$ is injective.  A continuous injection
from the compact space $Q$ into the Hausdorff space $Y$ is a topological
embedding.
\end{proof}

\begin{corollary}\label{cor:kernel-dichotomy}
Assume $\mathsf{DC}+\TP+\BCTWO$.  Let $E$ be an equivalence relation on
$2^\omega$ for which there is a map $f:2^\omega\to Y$ into a separable
metric space satisfying
\[
 xEy\quad\Longleftrightarrow\quad f(x)=f(y).
\]
Inside every perfect $P$ there is either a perfect subset of one $E$-class or
a perfect set of pairwise $E$-inequivalent points.  In particular, this
applies to every smooth Borel equivalence relation.
\end{corollary}

\begin{corollary}\label{cor:wo-codomain}
Assume $\mathsf{DC}+\TP+\BCTWO$.  If $Y$ is a well-orderable separable
metric space, then every map $f:2^\omega\to Y$ is constant on a perfect
subset of every prescribed perfect set.
\end{corollary}

\begin{proof}
The embedding alternative in \Cref{thm:constant-embedding} would inject a
perfect set into a well-orderable set and hence well-order that perfect set,
contrary to \Cref{lem:continuum-not-wo}.
\end{proof}

We finish the section by relating to Sacks forcing. Let $\mathbb S$ be Sacks forcing, viewed as the nonempty perfect subsets of $2^\omega$ ordered by inclusion, and denote
\[
 \mathbb Q_{\WO}=
 \bigl(\mathcal P(2^\omega)\smallsetminus\WO,\leq_{\WO}\bigr),
 \ \text{where }
 A\leq_{\WO}B\Longleftrightarrow A\smallsetminus B\in\WO,
\]
where $\WO$ denotes the ideal of well-orderable subsets of $2^\omega$.

\begin{lemma}\label{lem:perfect-quotient}
Assume $\TP+\BCTWO$.  If $P,Q\subseteq2^\omega$ are perfect, then
\[
 P\leq_{\WO}Q\quad\Longleftrightarrow\quad P\subseteq Q.
\]
Moreover, $P$ and $Q$ are compatible in $\mathbb Q_{\WO}$ if and only if
they have a common perfect subset.
\end{lemma}

\begin{proof}
Only the forward implication in the first assertion is nontrivial.  If
$x\in P\smallsetminus Q$, choose a basic neighborhood $[s]$ of $x$ disjoint from
the closed set $Q$.  The nonempty relatively open set $P\cap[s]$ contains a
perfect subset, contradicting $P\smallsetminus Q\in\WO$.

A common perfect subset is a common positive extension.  Conversely, suppose
that a positive set $A$ satisfies $A\leq_{\WO}P,Q$.  Then
\[
 A\smallsetminus(P\cap Q)\subseteq(A\smallsetminus P)\cup(A\smallsetminus Q)
\]
is well-orderable.  If $P\cap Q$ were also well-orderable, then $A$ would be
well-orderable, a contradiction.  Thus $P\cap Q$ is positive, and $\TP$
gives a perfect subset of the intersection.
\end{proof}

\begin{theorem}[Sacks quotient]\label{thm:sacks-quotient}
Assume $\TP+\BCTWO$.  The map $P\mapsto[P]_{\WO}$ is a dense embedding of
Sacks forcing into the separative quotient of $\mathbb Q_{\WO}$.  Hence
$\mathbb Q_{\WO}$ is forcing-equivalent to Sacks forcing.
\end{theorem}

\begin{proof}
Every positive set is non-well-orderable and therefore contains a perfect
subset, so perfect conditions are dense.  By \Cref{lem:perfect-quotient}, the
quotient order on perfect conditions is inclusion and incompatibility agrees
with Sacks incompatibility.  Passing to the separative quotient gives the
dense embedding.
\end{proof}

\section{Further consequences and questions}\label{sec:conclusion}

The smallness characterization immediately yields the positive-outer-measure consequence that motivated the paper.

\begin{theorem}\label{thm:positive-measure}
Let $\mu$ be any atomless Borel probability measure on $2^\omega$ in $M$.
If $A\subseteq2^\omega$ satisfies $\mu^*(A)>0$, then $A$ contains a perfect subset.
In particular, every set of positive Lebesgue outer measure contains a perfect subset.
\end{theorem}

\begin{proof}
If $A$ were well-orderable, \Cref{thm:smallness} would make it universally null, contradicting $\mu^*(A)>0$.
Hence $A$ is not well-orderable, so by \Cref{thm:truss} it contains a perfect subset.
\end{proof}

\begin{corollary}\label{cor:equiconsistency}
The following theories are equiconsistent:
\begin{enumerate}
\item $\mathsf{ZFC}$;
\item $\mathsf{ZF}+\mathsf{DC}$ together with the assertion that every set of positive Lebesgue outer measure contains a perfect subset.
\end{enumerate}
\end{corollary}

\begin{proof}
The Truss construction gives the forward consistency implication by \Cref{thm:model-principles,thm:positive-measure}.
\end{proof}

For comparisons with other perfect-set principles in models without choice, see Di Prisco and Galindo \cite{DiPriscoGalindo2010}.
The results above are restricted in two respects: the index-set characterizations concern subsets of $2^\omega$, and \Cref{prop:not-hurewicz} shows that the equivalence of the smallness properties considered here does not extend to the Hurewicz property.

We leave the following questions open.

\begin{question}
Over $\mathsf{ZF}+\mathsf{DC}$, what are the implication and consistency relations among $\TP$, $\BCTWO$, and the assertion $\WO=\SN$?
Do natural combinations of these statements imply a nontrivial fragment of choice such as $\mathsf{AC}_{\WO}$?
\end{question}

\begin{question}
Which other initial-segment symmetric models satisfy $\TP$, $\BCTWO$, or both?
Is there a forcing-theoretic condition on the coordinate forcing which implies $\RSWO$, and a separate condition which implies the perfect-set dichotomy?
\end{question}

\begin{question}
What are the Hurewicz subsets of Truss's model?
More generally, which selection properties adjacent to the Rothberger and Hurewicz properties in the Scheepers diagram admit an intrinsic description there?
\end{question}

\end{document}